\documentclass{amsart}

\usepackage{amsmath,amssymb,amsthm}
\usepackage{enumitem}
\usepackage{mathtools}
\usepackage[all,cmtip]{xy}

\newtheorem{theorem}{Theorem}[section]
\newtheorem{lemma}[theorem]{Lemma}
\newtheorem{corollary}[theorem]{Corollary}
\newtheorem{proposition}[theorem]{Proposition}
\newtheorem{conjecture}[theorem]{Conjecture}
\newtheorem{claim}[theorem]{Claim}

\theoremstyle{definition}

\newtheorem{definition}[theorem]{Definition}

\newtheorem{remark}[theorem]{Remark}

\newtheorem{question}[theorem]{Question}

\newtheorem*{acknowledgements}{Acknowledgements}

\numberwithin{equation}{section}

\DeclareMathOperator{\Alb}{Alb}
\DeclareMathOperator{\rat}{RatCurves}
\DeclareMathOperator{\Spec}{Spec}
\DeclareMathOperator{\Pos}{Pos}

\DeclareMathOperator{\Ker}{Ker}

\DeclareMathOperator{\Proj}{Proj}
\DeclareMathOperator{\pr}{pr}

\DeclareMathOperator{\rank}{rank}
\DeclareMathOperator{\id}{id}
\DeclareMathOperator{\Hilb}{Hilb}

\DeclareMathOperator{\im}{Im}
\DeclareMathOperator{\Sym}{Sym}
\DeclareMathOperator{\NE}{NE}

\newcommand{\cNE}{\overline{\NE}}
\newcommand{\red}{\mathrm{red}}

\newcommand{\cO}{\mathcal{O}}
\newcommand{\sE}{\mathcal{E}}
\newcommand{\sF}{\mathcal{F}}
\newcommand{\sG}{\mathcal{G}}
\newcommand{\sL}{\mathcal{L}}
\newcommand{\sM}{\mathcal{M}}

\newcommand{\sQ}{\mathcal{Q}}
\newcommand{\sK}{\mathcal{K}}
\newcommand{\sD}{\mathcal{D}}
\newcommand{\sU}{\mathcal{U}}

\newcommand{\bC}{\mathbb{C}}

\newcommand{\bG}{\mathbb{G}}

\newcommand{\bP}{\mathbb{P}}
\newcommand{\bQ}{\mathbb{Q}}
\newcommand{\bR}{\mathbb{R}}

\AtBeginDocument{%
   \def\MR#1{}
}

\setenumerate{
label=(\arabic*),
font=\normalfont
}

\begin{document}

\title[Nef tangent bundles in positive characteristic]{Projective varieties with nef tangent bundle in positive characteristic}

\author{Akihiro Kanemitsu}
\date{\today}
\address{Department of Mathematics, Graduate school of Science, Kyoto University, Kyoto 606-8502, Japan}
\email{kanemitu@math.kyoto-u.ac.jp}
\thanks{The first author is a JSPS Research Fellow and supported by the Grant-in-Aid for JSPS fellows (JSPS KAKENHI Grant Number 18J00681).}

\author{Kiwamu Watanabe}
\date{\today}
\address{Department of Mathematics, Faculty of Science and Engineering, Chuo University.
1-13-27 Kasuga, Bunkyo-ku, Tokyo 112-8551, Japan}
\email{watanabe@math.chuo-u.ac.jp}
\thanks{The second author is partially supported by JSPS KAKENHI Grant Number 17K14153, the Sumitomo Foundation Grant Number 190170 and Inamori Research Grants.}

\subjclass[2010]{14J40, 14J45, 14M17.}
\keywords{}

\begin{abstract}
Let $X$ be a smooth projective variety defined over an algebraically closed field of positive characteristic $p$ whose tangent bundle is nef.
We prove that $X$ admits a smooth morphism $X \to M$ such that the fibers are Fano varieties with nef tangent bundle and $T_M$ is numerically flat.
We also prove that extremal contractions exist as smooth morphisms.

As an application, we prove that, if the Frobenius morphism can be lifted modulo $p^2$, then $X$ admits, up to a finite \'etale Galois cover, a smooth morphism onto an ordinary abelian variety whose fibers are products of projective spaces.
\end{abstract}

\maketitle

\section{Introduction}

\subsection{Positivity of tangent bundles}
Given a variety $X$, there naturally exists an object $T_X$, called the \emph{tangent bundle} or the \emph{tangent sheaf} of $X$, which approximates $X$  linearly.
Thus the property of $T_X$ reflects the geometry of $X$, and conversely the (biregular) geometry of $X$ is restricted when a strong condition is supposed on $T_X$.
For example, celebrated Mori's proof of the Hartshorne conjecture says that the positivity of tangent bundle actually determines the isomorphic class of the variety $X$ in question; a smooth projective variety $X$ over an arbitrary algebraically closed field is isomorphic to a projective space $\bP^n$ if and only if the tangent bundle satisfies the positivity condition called \emph{ample} \cite{Mor79}.
This characterization of projective spaces is the algebro-geometric counter-part of the Frankel conjecture in complex geometry, which has been proved by \cite{SY80}.

Once this kind of characterization has been established, there are several attempts to generalize this type of results by posing weak positivity conditions on tangent bundles.
For example, over the field of complex numbers $\bC$, Campana and Peternell started the study of projective varieties satisfying a numerical semipositivity condition, called \emph{nef} (see Section~\ref{sect:preliminaries} for the definition).
Philosophically, if the tangent bundle is semipositive, then the variety is expected to decompose into the ``positive'' part and the ``flat'' part.
 Moreover the geometry of these two extremal cases is considered to be describable well.
Indeed, after the series of papers by Campana and Peternell \cite{CP91,CP93},
this type of decomposition has been accomplished by Demailly-Peternell-Schneider:

\begin{theorem}[{\cite[Main theorem]{DPS94}, $/\bC$}]
\label{thm:DPS:decomposition}
If the tangent bundle of a complex projective manifold $X$ is nef, then, up to an \'etale cover, the variety $X$ admits a fibration over an abelian variety whose fibers are Fano varieties (varieties with ample anti-canonical divisor $-K_X$).
\end{theorem}

This decomposition theorem reduces the study of varieties with nef tangent bundles to that of Fano varieties.
Conjecturally Fano varieties with nef tangent bundles are rational homogeneous varieties $G/P$, where $G$ is a semisimple algebraic group and $P$ is a parabolic group (Campana-Peternell conjecture \cite[11.2]{CP91}).

Main tools to study Fano varieties are the theory of rational curves and Mori's theory of extremal rays.
From the view point of Mori's theory, rational homogeneous spaces $G/P$ share an important feature that their contractions are always smooth.
In the same paper \cite{DPS94}, Demailly-Peternell-Schneider also studied varieties
with nef tangent bundle in view of Mori's theory, and showed the following fundamental structure theorem:

\begin{theorem}[{\cite[Theorem 5.2]{DPS94} and \cite[Theorem~4.4]{SW04}}, $/\bC$]
Let $X$ be a smooth complex projective variety and assume that $T_X$ is nef.
Then any  $K_X$-negative extremal contraction $f \colon X \to Y$ is smooth.
\end{theorem}

This smoothness theorem is a kind of evidence for the validity of the Campana-Peternell conjecture.
In fact, it plays an important role in the course of partial proofs of the Campana-Peternell conjecture \cite{CP91,CP93,Mok02,Hwa06,Wat14,Wat15,Kan16,Kan17,MOSW15,SW04}.
We also refer the reader to \cite{MOSWW15} for an account of the Campana-Peternell conjecture.

The purpose of this paper is to establish this type of decomposition theorem and also to study the structure theorem of varieties with nef tangent bundle in \emph{positive characteristic}.
In the rest of this section, $k$ denotes an algebraically closed field of positive characteristic $p >0$, and $X$ is a smooth projective variety defined over $k$.

\subsection{Rational curves on varieties with nef tangent bundle}
To establish the decomposition theorem, we need to distinguish the ``$K_X$-negative'' part and ``$K_X$-trivial'' part.
As is well-known, there are no rational curves on an abelian variety $A$, i.e.\ there are no nontrivial morphisms $\bP^1 \to A$.
On the other hand, there exists a rational curve on any Fano variety $X$ \cite[Theorem~5]{Mor79}.
Moreover, there are enough many rational curves so that any two points on $X$ can be connected by a chain of rational curves (Fano varieties are rationally chain connected) \cite{Cam92}, \cite[Theorem 3.3]{KMM92a}.
In \emph{characteristic zero}, rational chain connectedness of $X$ in turn implies that there exists a highly unobstructed rational curve $\bP^1 \to X$ called \emph{very free rational curve}, whose existence is equivalent to a more strong rational connectedness notion, called \emph{separable rational connectedness} of $X$ \cite[Theorem~2.1]{KMM92c} (see also Definition~\ref{def:RC}).

The first theorem of this article asserts that, if $T_X$ is nef, then the same also holds in positive characteristic:

\begin{theorem}[RCC $\implies$ SRC]\label{thm:RCC=>SRC}
Let $X$ be a smooth projective variety over $k$ and assume that $T_X$ is nef.
If $X$ is rationally chain connected, then $X$ is separably rationally connected.
\end{theorem}

Thus, if $X$ is RCC, then $X$ contains a very free rational curve, which is a highly unobstructed object.
By \cite[Corollaire 3.6]{Deb03}, \cite[Corollary 5.3]{She10}, \cite{BDS13}, \cite{Gou14}, we have the next corollary:

\begin{corollary}\label{cor:SRC}
Let $X$ be a smooth projective variety over $k$ and assume that $T_X$ is nef.
 If $X$ is rationally chain connected, then the following hold:
\begin{enumerate}
 \item $X$ is algebraically simply connected;
 \item $H^1(X, \cO_X) = 0$;
 \item every numerically flat vector bundle on $X$ is trivial.
\end{enumerate}
\end{corollary}

\subsection{Contractions of extremal rays}
Once the above theorems on rational curves are well-established, one can conduct a detailed study of varieties with nef tangent bundle in view of Mori theory.
Recall that, over arbitrary algebraically closed field $k$, Mori's cone theorem holds \cite{Mor82}.
Namely, the Kleiman-Mori cone $\cNE (X)$ is locally polyhedral in $K_X$-negative side and thus it decomposes as follows:
\[
\cNE(X) = \cNE _{K_X \geq 0} + \sum R_{i},
\]
where $R_i = \bR_{\geq0}[C_i]$ are $K_X$-negative extremal rays, each of which is spanned by a class of a rational curve $C_i$.
Note that, \emph{in characteristic zero}, each extremal ray is realized in a geometrical way, i.e.\ there exists the contraction of each extremal ray $R$ (see e.g.\ \cite{KM98}), while the existence of extremal contractions is widely open in \emph{positive characteristic}.
The following theorem asserts that, if $T_X$ is nef, then the contraction of an extremal ray exists and, in fact,  it is smooth:

\begin{theorem}[Existence and smoothness of contractions]
\label{thm:contraction}
Let $X$ be a smooth projective variety over $k$ and assume that $T_X$ is nef.

Let $R \subset \cNE (X)$ be a $K_X$-negative extremal ray.
Then the contraction $f \colon X \to Y$ of $R$ exists and the following hold:
\begin{enumerate}
 \item $f$ is smooth;
 \item any fiber $F$ of $f$ is an SRC Fano variety with nef tangent bundle.
 \item $T_Y$ is again nef.
\end{enumerate}
\end{theorem}

See also Corollary~\ref{cor:face} for contractions of extremal \emph{faces}.

\subsection{Decomposition theorem}
In characteristic zero, a projective variety $X$ with nef tangent bundle admits the Demailly-Peternell-Schneider decomposition (Theorem~\ref{thm:DPS:decomposition}).
In fact $X$ itself admits a smooth Fano fibration $\varphi \colon X \to M$ over a projective variety $M$ with numerically flat tangent bundle $T_M$.
Thus this morphism $\varphi$ contracts any rational curve on $X$ to a point.
Hence this morphism is the maximally rationally connected fibration (MRC fibration) \cite{Cam92}, \cite{KMM92c}.
Therefore the decomposition is obtained by considering the MRC fibration of $X$.
One can expect that a similar picture also holds in positive characteristic.
Note that, in positive characteristic, due to the absence of generic smoothness, the MRC fibration is not an appropriate object.
The substitute of such a fibration is the maximally rationally chain connected fibration (MRCC fibration) \cite[Chapter IV, Section 5]{Kol96}.
The following theorem ensures that each fiber of the MRCC fibration is a Fano variety:

\begin{theorem}[RCC $\implies$ Fano]
\label{thm:RCC=>Fano}
Let $X$ be a smooth projective variety over $k$ and assume that $T_X$ is nef.
If $X$ is rationally chain connected, then $X$ is a smooth Fano variety.
Moreover, the Kleiman-Mori cone $\NE(X)$ of $X$ is simplicial.
\end{theorem}

 In fact, the above assertion also holds in relative settings (see Theorem~\ref{thm:simplicial}).

By combining the above theorems we can obtain the decomposition theorem of varieties with nef tangent bundle:

\begin{theorem}[Decomposition theorem]
\label{thm:decomposition}
Let $X$ be a smooth projective variety over $k$ and assume that $T_X$ is nef.
Then $X$ admits a smooth contraction $\varphi \colon X \to M$ such that
\begin{enumerate}
 \item $\varphi$ is the MRCC fibration of $X$;
 \item any fiber of $\varphi$ is a smooth SRC Fano variety with nef tangent bundle;
 \item $T_M$ is numerically flat.
\end{enumerate}
\end{theorem}

\subsection{Questions}
The above decomposition theorem reduces the study of varieties with nef tangent bundle to two cases (the Fano case and the $K_X$-trivial case).
The following suggests the possible structures of varieties in these two cases:

\begin{question}[{\cite[Conjecture~11.1]{CP91}, \cite[Question 1]{Wat17}}]
Let $X$ be a smooth projective variety over $k$ and assume $T_X$ is nef.
\begin{enumerate}
 \item If $X$ is a Fano variety, then is $X$ a homogeneous space $G/P$, where $G$ is a semisimple algebraic group and $P$ is a parabolic subgroup?
 \item If $K_X \equiv 0$ or, equivalently, $T_X$ is numerically flat, then is $X$ an \'etale quotient of an abelian variety? 
\end{enumerate}
\end{question}

Note that, in characteristic zero, the second assertion is true by virtue of the Beauville-Bogomolov decomposition, which is studied intensively in characteristic positive \cite{PZ20}, but is still open.
See also \cite{Wat17,MS87} for partial answers on the above question. 

\subsection{Application}
In the last part of this paper, we will apply our study of varieties with nef tangent bundle to the study of \emph{$F$-liftable varieties}.
A smooth projective variety $X$ (over an algebraically closed field of positive characteristic) is said to be $F$-liftable, if it lifts modulo $p^2$ with the Frobenius morphism (see Definition~\ref{def:f-liftable} for the precise definition).
Natural examples of $F$-liftable varieties are toric varieties and ordinary abelian varieties.
Conversely, any $F$-liftable variety is expected to decompose into these two types of varieties:

\begin{conjecture}[{\cite[Conjecture~1]{AWZ17}}]
Let $Z$ be an $F$-liftable variety.
Then there exists a finite Galois cover $f \colon Y \to Z$ such that the Albanese morphism $\alpha_{Y} \colon Y \to \Alb (Y)$ is a toric fibration.
\end{conjecture}

This conjecture is confirmed in a case:
In \cite{BTLM97} and \cite{AWZ17}, the case of homogeneous varieties was solved.
In fact, in \cite[Proposition~6.3.2]{AWZ17}, the conjecture is checked when $X$ is a Fano manifold with nef tangent bundle whose Picard number is one.
Here we apply our study to show that the conjecture holds under a more general situation that $T_X$ is nef:

\begin{theorem}[$F$-liftable varieties with nef tangent bundle]\label{thm:F-liftable} 
If $X$ is $F$-liftable and $T_X$ is nef, then there exists a finite \'etale Galois cover $f \colon Y \to X$ such that the Albanese morphism $\alpha_{Y} \colon Y \to \Alb (Y)$ is a smooth morphism onto an ordinary abelian variety whose fibers are products of projective spaces.
\end{theorem}

\subsection{Outline of the paper}
This article is organized as follows:
In Section~\ref{sect:preliminaries}, we provide a preliminaries on nef vector bundles, and recall some basic properties on nef vector bundles.

In Section~\ref{sect:SRC}, we study the separable rational connectedness of varieties with nef tangent bundle and prove Theorem~\ref{thm:RCC=>SRC}.
The main ingredient of the proof is Shen's theorem that provides a relation between separable rational connectedness and foliations in positive characteristic \cite{She10}.
By using this relation, we construct a purely inseparable finite morphism $X \to Y$ (if $X$ is RCC) such that $Y$ also has nef tangent bundle and is SRC.
If $X \not \simeq Y$, then we can find a nowhere vanishing vector field $D$ on $Y$ and moreover, by choosing $D$ suitably, we can construct an action of the group scheme $G = \mu_p$ or $\alpha_p$ on $Y$ without fixed points.
Then, by following Koll\'ar's proof of simple connectedness of SRC varieties \cite[Corollaire 3.6]{Deb03}, we will have a contradiction and hence $X$ itself is SRC.

In Section~\ref{sect:contractions}, we study extremal contractions on varieties with nef tangent bundle and prove Theorem~\ref{thm:contraction}.
Theorem~\ref{thm:contraction} essentially follows from the arguments of \cite[Theorem 2.2]{Kan18Kequiv} and \cite[Lemma~4.12]{SW04}, while there are some problems to adapt these arguments in our case.
The most major issue is due to pathological phenomena in positive characteristic.
We overcome this issue by using Theorem~\ref{thm:RCC=>SRC}.

In Section~\ref{sect:Fano}, we will prove Theorem~\ref{thm:RCC=>Fano}.
In fact, we will study the relative Kleiman mori cone $\NE (X/Y)$ of a contraction $f \colon X \to Y$ with RCC fibers.
The main theorem of this section is Theorem~\ref{thm:simplicial}, which proves that the cone $\NE (X/Y)$ is simplicial cone spanned by $K_X$-negative extremal rays.
The proof essentially goes as that of \cite[Proof of Theorem~4.16]{Wat20}.

In Section~\ref{sect:decomposition}, we prove Theorem~\ref{thm:decomposition}, which easily follows from the theorems in previous sections.

In the last section, we study the case where $X$ is $F$-liftable and prove Theorem~\ref{thm:F-liftable}.

\subsection*{Conventions}
Throughout this paper, we work over an algebraically closed field $k$ of characteristic $p>0$.
We use standard notations and conventions as in \cite{Har77}, \cite{Kol96}, \cite{KM98} and \cite{Deb01}:
\begin{itemize}
 \item Unless otherwise stated, a \emph{point} means a closed point and a \emph{fiber} means a fiber over a closed point.
 \item  A \emph{rational curve on a variety $X$} is a nonconstant morphism $f \colon \bP^1 \to X$ or, by an abuse of notation, its image $f(\bP^1) \subset X$.
 \item A \emph{contraction} $f \colon X \to Y$ is a projective morphism of varieties such that $f_* \cO_X =\cO _Y$.
 \item For a contraction $f \colon X \to Y$, we denote by $N_1(X / Y)$ the $\bR$-vector space of the numerical equivalence classes of relative $1$-cycles.
 The \emph{cone of relative effective $1$-cycles} $\NE (X/Y) \subset N_1(X/Y)$ is the semi-subgroup generated by the classes of effective $1$-cycles.
 The closure $\cNE(X/Y)$ is called the \emph{relative Kleiman-Mori cone}. 
 \item For a contraction $X \to Y$, the \emph{relative Picard number}  $\rho(X /Y)$ is the rank of $N_1(X/Y)$.
 \item For a contraction $f \colon X \to Y$, $N^1(X/Y)$ is the $\bR$-vector space generated by the $f$-numerical equivalence classes of Cartier divisors.
 Note that $N^1(X/Y)$ and $N_1(X/Y)$ are finite dimensional vector spaces, which are dual to each other.
 \item If $Y =\Spec k$, then we will use $N_1(X)$, $\NE(X)$, $\rho(X)$, $N^1(X)$ instead of  $N_1(X/Y)$, $\NE(X/Y)$, $\rho(X/Y)$, $N^1(X/Y)$.
 \item For a smooth projective variety $S$, a \emph{smooth $S$-fibration} is a smooth morphism between varieties whose closed fibers are isomorphic to $S$.
 \item $\sE^{\vee}$ is the dual vector bundle of a vector bundle $\sE$.
 \item $\bP(\sE)$ is the Grothendieck projectivization $\Proj(\Sym \sE)$ of a vector bundle $\sE$.
 \item A subsheaf $\sF \subset \sE$ of a locally free sheaf $\sE$ is called a \emph{subbundle} if the quotient bundle $\sE/\sF$ is a locally free sheaf.
 \item For a projective variety $X$, we denote by $F_X$ the absolute Frobenius morphism $X \to X$.
\end{itemize}

Note that a contraction, if it exists, is uniquely determined by $\NE (X/Y) \subset N_1(X)$.
A \emph{contraction of a $K_X$-negative extremal ray $R$} is, by definition, a contraction $f \colon X \to Y$ such that $\rho (X/Y) = 1$  and $R \coloneqq \NE(X/Y) \subset N_1(X/Y)_{K_X  \leq 0}$.

Furthermore we use standard terminology on families of rational curves.
For example,
\begin{itemize}
 \item $\rat^n (X)$ is the scheme that parametrizes rational curves on $X$, and a \emph{family of rational curves} is an irreducible component $\sM$ of the scheme $\rat^n (X)$.
 \item For a family $\sM$ of rational curves, there exists the following diagram:
 \[
 \xymatrix{
 \sU \ar[d]_{p} \ar[r]^{q} &X \\
 \sM,
 }
 \]
where
\begin{itemize}
 \item $p$ is a smooth $\bP^1$-fibration, which corresponds to the universal family;
 \item $q$ is the evaluation map.
\end{itemize}
Thus a point $m \in \sM$ corresponds to a rational curve $q|_{p^{-1}(m)} \colon p^{-1}(m) \simeq \bP^1 \to X$.
\item A family $\sM$ of rational curves is called \emph{unsplit} if it is projective.
\end{itemize}
For the details of the constructions of these spaces and for basic properties, we refer the reader to \cite[Chapter~II, Section~2]{Kol96} and also to \cite{Mor79}.

\begin{acknowledgements}
The content of this article was firstly planned as a collaborative work of these authors with Doctor Sho Ejiri, who has refrained from being listed as a coauthor for the reason that he did not think his contribution was enough.
Nevertheless we are greatly indebted to him for sharing ideas and discussing about the subject, especially about the case of numerically flat tangent bundles and about the Frobenius splitting methods; we wish to express our sincere gratitude to him.

The first author would like to thank Doctor Tatsuro Kawakami for helpful discussions.
\end{acknowledgements}

\section{Preliminaries}\label{sect:preliminaries}

\subsection{Nef vector bundles}
We collect some facts on nef vector bundles.

Let $\sE$ be a vector bundle on a smooth projective variety $X$.
Then the bundle $\sE$ is called \emph{nef} if the relative tautological divisor $\cO_{\bP(\sE)}(1)$ is nef.
By the definition, we see that any quotient bundle of a nef vector bundle is again nef.
Also we see that, for a morphism $f \colon Y \to X$ from a projective variety $Y$, the pullback $f^*\sE$ is nef if $\sE$ is nef, and the converse holds if $f$ is surjective.
A vector bundle $\sE$ is called \emph{numerically flat} if $\sE$ and its dual $\sE^{\vee}$ are nef.

Note that, by \cite[Proposition~3.5]{Bar71},
the tensor product of two nef vector bundles are nef again. In particular, if $\sE$ is nef, then the exterior products of $\sE$ are also nef.

By the same argument as in the case of characteristic zero (cf.\ \cite[Theorem~6.2.12]{Laz04b} and \cite[Proposition~1.2]{CP91}), we have the following:

\begin{proposition}[Nef vector bundles]\label{prop:nef_bundle}
Let $\sE$ be a vector bundle on a smooth projective variety $X$. We have the following:
\begin{enumerate}
 \item $\sE$ is numerically flat if and only if  both $\sE$ and $\det \sE^{\vee}$ are nef.
 \item\label{prop:nef_bundle4} Let $0 \to \sF \to \sE \to \sG \to 0 $ be an exact sequence of vector bundles.
 Assume that $\sF$ and $\sG$ are nef. Then $\sE$ is also nef.

 Conversely, if $\sE$ is nef and $c_1(\sG) \equiv 0$, then $\sF$ is nef.
 \item\label{prop:nef_bundle5} Assume $\sE$ is nef.
 If $\sL \to \sE^\vee$ is a non-trivial morphism from a numerically trivial line bundle, then $\sL$ defines a subbundle of $\sE^\vee$.
\end{enumerate}
\end{proposition}

\begin{lemma}[Numerically flat quotient bundles]\label{lem:num_flat_quot}
Let $\sE$ be a  nef vector bundle on a smooth  projective variety, and $\sQ$ a torsion free quotient of $\sE$ such that $c_1(\sQ) \equiv 0$.
Then $\sQ$ is a numerically flat vector bundle.
Moreover the kernel of $\sE \to \sQ$ is a nef vector bundle.
\end{lemma}

\begin{proof}
Consider the dual map $\sQ^{\vee} \to \sE^{\vee}$.
Then the map $\det (\sQ^{\vee}) \to \bigwedge^{\rank \sQ} \sE^\vee$ is a bundle injection by Proposition~\ref{prop:nef_bundle}~\ref{prop:nef_bundle5}. 
Then, by \cite[Lemma 1.20]{DPS94},  the sheaf $\sQ^{\vee}$ is a subbundle of $\sE^{\vee}$.
Hence the composite $\sE \to \sQ \to \sQ^{\vee \vee}$ is surjective.
In particular, the map $\sQ \to \sQ^{\vee \vee}$ is also surjective.
Since $\sQ$ is torsion free, we have $\sQ = \sQ^{\vee \vee}$.
Now it follows from Proposition~\ref{prop:nef_bundle}~\ref{prop:nef_bundle4} that the kernal of $\sE \to \sQ$ is a nef vector bundle.
\end{proof}

\begin{theorem}[{\cite{BDS13}}]\label{thm:SRC:bundle:trivial}
Let $X$ be a smooth projective separably rationally connected variety and $\sE$ a vector bundle on $X$.
Assume that, for any rational curve $f\colon  \bP^1 \to X$, the pull-back $f^{\ast}\sE$ is trivial.
Then $\sE$ itself is trivial.
\end{theorem}

This implies the following:
\begin{corollary}[{see \cite{Gou14}}]\label{cor:SRC:h1}
For a smooth projective separably rationally connected variety $X$, the first cohomology $H^1(X, \cO_X)$ vanishes.
\end{corollary}

\section{Separable rational connectedness}\label{sect:SRC}
In this section, we will prove Theorem~\ref{thm:RCC=>SRC} and Corollary~\ref{cor:SRC}.

\subsection{Preliminaries: Foliations and separable rational connectedness}
Here we collect several results from \cite{She10}, which describe the relation between separable rational connectedness of a variety $X$ and foliations on $X$.
For an account of the general theory of foliations, we refer the reader to \cite{Eke87,MP97}.

Let $X$ be a normal projective variety of dimension $n$.
A rational curve $f \colon \bP^1 \to X$ is called \emph{free} (resp.\ \emph{very free}) if $f(\bP^1)$ is contained in the smooth locus of $X$, and $f^*T_X$ is nef (resp.\ ample).

\begin{definition}[{RCC, RC, FRC, SRC {\cite[Chapter~IV. Definition~3.2]{Kol96}, \cite[Definition 1.2]{She10}}}]\label{def:RC}
Let $X$ be a normal projective variety over $k$.
Then $X$ is called
\begin{enumerate}
 \item \emph{rationally chain connected (RCC)} if there exist a variety $T$ and a scheme $\sU$ with morphisms $(T \xleftarrow{p} \sU \xrightarrow{q} X)$ such that
 \begin{itemize}
  \item $p$-fibers are connected proper curves with only rational components;
  \item the natural map $q^{(2)} \colon \sU \times _T \sU \to X \times X$ is dominant.
 \end{itemize}
 
 \item \emph{rationally connected (RC)} if there exist a variety $T$ and a scheme $\sU$ with morphisms $(T \xleftarrow{p} \sU \xrightarrow{q} X)$ such that
 \begin{itemize}
  \item $p$-fibers are \emph{irreducible} rational curves;
  \item the natural map $q^{(2)} \colon \sU \times _T \sU \to X \times X$ is dominant.
 \end{itemize}
 
 \item  \emph{freely rationally connected (FRC)} if there exists a variety $T$ with morphisms $(T \xleftarrow{\pr_2} \bP^1 \times T \xrightarrow{q} X)$ such that
 \begin{itemize}
  \item each $\pr_2$-fiber defines a \emph{free} rational curve on $X$;
  \item the natural map $q^{(2)} \colon \bP^1 \times \bP^1 \times T  \to X \times X$ is dominant.
 \end{itemize}

 \item \emph{separably rationally connected (SRC)} if there exists a variety $T$  with morphisms $(T \xleftarrow{\pr_2} \bP^1 \times T \xrightarrow{q} X)$ such that
 \begin{itemize}
  \item the natural map $q^{(2)} \colon \bP^1 \times \bP^1 \times T  \to X \times X$ is dominant and smooth at the generic point.
 \end{itemize}
\end{enumerate}
\end{definition}

\begin{remark}
\hfill
\begin{enumerate}
 \item Any smooth Fano variety is rationally chain connected (see for instance \cite[Chapter~V. Theorem~2.13]{Kol96}).
 \item Note that the definition of separable rational connectedness here is slightly different from the definition in \cite{She10} (the existence of very free rational curves).
 However, if $X$ is smooth, then the separable rational connectedness of $X$ is equivalent to the existence of a very free rational curve on $X$ (see for instance \cite[Chapter~IV. Theorem~3.7]{Kol96}).
 Thus the two definitions coincide. Note also that, in the situation of this paper, we mainly consider smooth varieties and thus we do not need to care about the differences (cf. the smoothness of $X^{[1]}$ in Proposition~\ref{prop:num_flat_subbundle}).
 \item In general, the following implications hold:
 \[
 \text{SRC} \implies \text{FRC} \implies \text{RC} \implies \text{RCC}.
 \]
 Assume that $X$ is smooth and $T_X$ is nef.
 Then $\text{RCC} \implies \text{FRC}$.
 This follows from the smoothing technique of free rational curves (see e.g. \cite[Chapter~II. 7.6]{Kol96}).
\end{enumerate}
\end{remark}

Assume that a normal projective variety $X$ contains a free rational curve $f \colon \bP^1 \to X$.
Since any vector bundle on $\bP^1$ is a direct sum of line bundles,
we have
 \[
 f^* T_X \simeq \bigoplus_{i=1}^{r} \cO(a_i) \bigoplus \cO^{\oplus n-r}
 \]
with $a_i >0$.
Then the positive part $ \bigoplus_{i=1}^{r} \cO(a_i)$ defines a subsheaf of $f^*T_X$ that is independent of the choice of decomposition of $f^*T_X$.

\begin{definition}[Positive ranks and maximally free rational curves {\cite[Definition 2.1]{She10}}]
Let $f \colon \bP^1 \to X$ be a free rational curve and consider a decomposition of $f^* T_X$ as above:
 \[
 f^* T_X \simeq \bigoplus_{i=1}^{r} \cO(a_i) \bigoplus \cO^{\oplus n-r}.
 \]
\begin{enumerate}
 \item The ample subsheaf $ \bigoplus_{i=1}^{r} \cO(a_i) \subset f^*T_X$ is called the \emph{positive part} of $f^*T_X$, and denoted by $\Pos (f^*T_X)$.
 \item $r = \rank \Pos (f^*T_X)$ is called the \emph{positive rank} of the free rational curve $f$.
 \item The \emph{positive rank of $X$} is defined as the maximum of positive ranks of the free rational curves.
 \item A free rational curve $f\colon \bP^1 \to X$ is called \emph{maximally free} if its positive rank is the positive rank of $X$.
\end{enumerate}
\end{definition}

\begin{proposition}[{\cite[Proposition 2.2]{She10}}]
Let $x \in X$ be a closed point.
Assume that there exists a maximally free rational curve $f \colon \bP^1 \to X$ with $f(\bP^1) \ni x$.
Then there exists a $k$-subspace $\sD (x) \subset  T_X\otimes k(x)$ such that 
\begin{itemize}
 \item for every maximally  free rational curve $f \colon \bP^1 \to X$ with $f(0) = x$, we have
 \[
 \Pos(f^*T_X)\otimes k(0) = \sD (x)
 \]
 as a subspace of $T_X \otimes k(x)$.
\end{itemize}
\end{proposition}

The next proposition ensures that these subspaces
$\sD (x)$ patch together to define a subsheaf of $T_X$:

\begin{proposition}[{\cite[Proposition 2.5]{She10}}]
There exist a nonempty open subset $U \subset X$ and a subbundle $\sD  \subset T_U$ such that
\[
\sD\otimes k(x)  = \sD (x)
\]
holds for all $x \in U$.
\end{proposition}

In the following we denote also by $\sD  \subset T_X$ the saturation of $\sD $ in $T_X$ (by an abuse of notation).

\begin{theorem}[{\cite[Proposition 2.6]{She10}}]
The sheaf $\sD $ defines a foliation. Namely, $\sD$ satisfies the following conditions:
\begin{enumerate}
 \item $\sD$ is involutive, i.e. it is closed under the Lie bracket $[\sD ,\sD ] \subset \sD $;
 \item $\sD$ is $p$-closed, i.e. $\sD ^p \subset \sD $.
\end{enumerate}
\end{theorem}

Then there exist a variety $X^{[1]} \coloneqq X/\sD $, which is called the quotient of $X$ by the foliation $\sD$, and a sequence of morphisms
\[
X \xrightarrow{f} X^{[1]} \xrightarrow{g} X^{(1)} \xrightarrow{\sigma} X
\]
such that
\begin{itemize}
 \item $f \colon X \to X^{[1]}$ is the quotient map induced from the foliation;
 \item $\sigma \circ g \circ f$ is the absolute Frobenius morphism $F_{X}$;
 \item $g \circ f$ is the relative Frobenius morphism over $k$.
\end{itemize}
Note that $f$ and $g$ are morphisms over $\Spec(k)$, but $\sigma$ is not.
Note also that, in general, $X^{[1]}$ is normal.
In the situation of this paper,  however, $\sD $ is always a subbundle and hence $X^{[1]}$ is smooth (see Proposition~\ref{prop:num_flat_subbundle})

\begin{theorem}[{\cite[Proposition 3.4 and Theorem 5.1]{She10}}]\label{thm:Shen}
The following hold: 
\begin{enumerate}
 \item $X$ contains a very free rational curve if and only if $X^{[1]} \simeq  X^{(1)}$.
 \item If $X$ is FRC, then so is $X^{[1]}$.
 \item Set $X^{[m+1]} \coloneqq (X^{[m]})^{[1]}$ inductively.
 Then, if  $X$ is FRC, then $X^{[m]}$ contains a very free rational curve for $m \gg 0$.
\end{enumerate}
\end{theorem}

\begin{remark}[Foliations and tangent bundles]\label{rem:foliation}
Recall that $\sigma \colon X^{(1)} \to X$ is the base change of the absolute Frobenius $F_{\Spec(k)}$ on $\Spec(k)$ by the structure morphism $X \to \Spec(k)$ (Note that $F_{\Spec(k)}$ is an isomorphism).
In particular, there is a natural identification $\sigma ^{*}T_X = T _{X^{(1)}}$.
Then $\sigma^* \sD \subset \sigma ^{*}T_X = T _{X^{(1)}}$ defines an $\cO_{X^{(1)}}$-subsheaf.

Furthermore, by considering the pull-back by $g$, we obtain a morphism $g^*\sigma ^* \sD  \to g^*\sigma ^{*}T_X = g^*T _{X^{(1)}}$, which gives a subbundle $g^* \sigma ^* \sD \subset g^* \sigma ^{*}T_X = T _{X^{(1)}}$ over the open subset on which $\sD $ is a \emph{subbundle}.
Note that the differential $dg \colon T_{X^{[1]}} \to g^* T_{X^{(1)}}$ factors through $g^* \sigma ^* \sD $.
Denote by $\sK$ the kernel of the map $T_{X^{[1]}} \to g^* \sigma ^* \sD $.
Then we have $f^* \sK \simeq T_X/\sD $.
Moreover, we have the following exact sequence on the open subset where $\sD $ is a subbundle:
\begin{itemize}
 \item $ 0 \to \sK \to T_{X^{[1]}} \to g^*\sigma^* \sD  \to 0$;
 \item $ 0 \to \sD  \to T_X \to f^* T_{X^{[1]}} \to F_{X}^* \sD  \to 0 $.
\end{itemize}

Note that $\sK$ is the foliation corresponds to $X^{[1]} \to  X^{(1)}$, and hence $\sK$ is involutive and $p$-closed.
\end{remark}

\subsection{Action of group schemes $\mu_p$ and $\alpha_p$}
Here we briefly recall the relation between vector fields and actions of group schemes $\mu_p$ and $\alpha_p$. See e.g.\ \cite[Chapitre II, \S7]{DG70}, \cite[\S1]{RS76}, \cite[\S1]{MN91} for an account.

\begin{definition}[$\mu_p$ and $\alpha_p$]
The group schemes $\mu_p$ and $\alpha_p$  are defined as follows:
\begin{itemize}
 \item $\mu_p \coloneqq \Spec k[x]/(x^p -1) \subset \bG_m = \Spec k[x,x^{-1}]$;
 \item $\alpha_p \coloneqq \Spec k[x]/(x^p) \subset \bG_a = \Spec k[x]$.
\end{itemize}
Here the group scheme structures are induced from $\bG_m$ and $\bG_a$.
\end{definition}

\begin{proposition}[$p$-closed vector fields]
Let $X$ be a smooth variety and $D \in H^0(T_X)$ be a vector field.
\begin{enumerate}
 \item There is a bijection between the set of vector fields $D$ with $D^p = D$ and the set of $\mu _p$-actions on $X$.
 \item There is a bijection between the set of vector fields $D$ with $D^p = 0$ and the set of $\alpha _p$-actions on $X$.
\end{enumerate}
\end{proposition}

\begin{remark}[Quotients by $\mu_p$ and $\alpha_p$]
Let $X$ be a smooth projective variety and $D$ a vector field with $D^p = D$ or $D^p = 0$.
Then there exists the action of $G = \mu_p$ or $\alpha_p$ corresponding to $D$.
Then, by \cite[\S12, Theorem 1]{Mum70}, there exists the quotient $X/G$.
Note that the map $X \to X/G$ corresponds to the foliation spanned $D$.
\end{remark}

\begin{lemma}[{cf.\ \cite[\S1 Lemma 1 and Corollary]{RS76}}]\label{lem:p-closed}
Let $V \subset H^0(T_X)$ be an involutive and $p$-closed $k$-subspace.
Then there exists a vector field $D \in V$ such that $D^p =D$ or $D^p = 0$.
\end{lemma}

\begin{definition}[Fixed points]
Set $G = \mu_p$ or $\alpha_p$ and consider an action of $G$ on a smooth variety $X$.
Let $D$ be the corresponding vector field.
Then the \emph{set of $G$-fixed points} is defined as the zero locus of $D$.
\end{definition}

\begin{remark}
A vector field $D \in H^0 (T_X)$ defines a morphism $\varphi_D \colon \cO \to T_X$.
Then the dual of this map $\varphi_D^{\vee} \colon \Omega \to \cO_X$ gives a derivation $f_D \coloneqq \varphi ^{\vee} \circ d$.
Since the sheaf $\Omega_X$ is generated by $d\cO_X$ as an $\cO_X$-module, the following three subschemes are the same:
\begin{enumerate}
 \item The zero locus of $D$.
 \item The closed subscheme whose ideal is $\im (\varphi_D^{\vee})$.
 \item The closed subscheme whose ideal is generated by $\im (f_D)$.
\end{enumerate}
\end{remark}

\begin{lemma}[{\cite[Section~1, Lemma 2]{RS76}}]
There exists a $\mu_p$-action on $\bP^1$.
Moreover, for any $\mu_p$-action on $\bP^1$,  the set of $\mu_p$-fixed points consists of two points.

Similarly, there exists an $\alpha_p$-action on $\bP^1$.
Moreover, for any $\alpha_p$-action on $\bP^1$,  the support of the set of $\alpha_p$-fixed points consists of one point.
\end{lemma}

\subsection{Proof of Theorem~\ref{thm:RCC=>SRC}}
Throughout this subsection, we assume that
\begin{itemize}
 \item $X$ is a smooth projective variety with nef tangent bundle and moreover $X$ is rationally chain connected.
\end{itemize}

\begin{lemma}
$N_1(X)$ is spanned by maximally free rational curves.
\end{lemma}

\begin{proof}
It is enough to show that a divisor $H$ is numerically trivial if $H \cdot C = 0$ for any maximally free rational curve $C$.

Note that, by \cite[Chapter~IV. Theorem 3.13]{Kol96}, $N_1 (X)$ is generated by rational curves.
Assume $H \not \equiv 0$. Then there exists a rational curve $g \colon \bP^1 \to D \coloneqq g(\bP^1) \subset X$ such that $H \cdot D \neq 0$.
Since the tangent bundle is nef, the curve $D$ is free.

Let $f \colon \bP^1 \to C \coloneqq f(\bP^1) \subset X$ be a maximally free rational curve.
Since $C$ and $D$ are both free rational curves, deformations of these curves cover open subsets of $X$ and hence we may assume that $C \cap D \neq \emptyset$.
Then, by \cite[Chapter~II. Theorem 7.6]{Kol96}, there is a smoothing $C'$ of $C \cup D$.

\begin{claim}
 $C'$ is a maximally free rational curve.
\end{claim}

\begin{proof}[Proof of Claim]
Let $f ' \colon \bP^1 \to C' \subset X$ be the normalization of $C'$, and $h \colon Z \to C \cup D \subset X$ a map from the tree of two rational curves to $C\cup D$.
Then, by the semicontinuity, we have
\[
h^0(f'^*\Omega_X) \leq h^0(h^*\Omega_X).
\]

Let $Z_1$ and $Z_2$ be the irreducible components of $Z$, each of which is isomorphic to $\bP^1$, and $p \in Z_1 \cap Z_2$ be the intersection point.
We may assume $h|_{Z_1} = f$ and $h|_{Z_2} = g$. 
Then we have the following exact sequence:
\[
0 \to g^*\Omega_X (-p) \to h^*\Omega_X \to f^* \Omega_X \to 0.
\]
Since $T_X$ is nef, we have $H^0 (g^*\Omega_X (-p)) =0$ and hence
$h^0 (h^*\Omega_X) \leq h^0 (f^*\Omega_X)$.
Thus we have $h^0 (f'^*\Omega_X) \leq h^0 (f^*\Omega_X)$ and have
\[
\rank \Pos (f') = n -  h^0 (f'^*\Omega_X) \geq n- h^0 (f^*\Omega_X) =\rank  \Pos(f).
\]
Since $f \colon \bP^1 \to X$ is maximally free, so is $f'$.
\end{proof}

Now, $H \cdot C' = 0$ by the assumption on $H$.
On the other hand, $H \cdot C' = H\cdot C + H \cdot D \neq 0$ and we obtain a contradiction.
\end{proof}

\begin{proposition}[Purely inseparable modification]\label{prop:num_flat_subbundle}
 $X^{[1]}$ is also a smooth, rationally chain connected, projective variety with nef tangent bundle.
Moreover, if $X$ is not separably rationally connected, then the subsheaf $\sK \subset T_{X^{[1]}}$ is a numerically flat non-trivial subbundle.
\end{proposition}

\begin{proof}
Consider the exact sequence
\[
0 \to \sD  \to T_X \to T_X/\sD  \to 0.
\]
By definition,  $T_X/\sD $ is trivial on the maximally free rational curves.
Note that, since $T_X$ is nef, any rational curve $C$ is free and hence it admits a deformation $C'$ such that $T_X/D$ is locally free on $C'$ \cite[Chapter~II. Proposition 3.7]{Kol96}.
In particular, $c_1(T_X/\sD )$ is nef on any rational curve.
By the previous lemma, we have $c_1(T_X/\sD ) \equiv 0$.
Then, by Lemma \ref{lem:num_flat_quot}, $T_X/\sD $ is a numerically flat vector bundle.
In particular, $X^{[1]}$ is also \emph{smooth}.
Note that $\sD $ is nef by Lemma~\ref{lem:num_flat_quot}.

Recall that there are the following exact sequences
\begin{itemize}
 \item $ 0 \to \sK \to T_{X^{[1]}} \to g^*\sigma^* \sD  \to 0$;
 \item $ 0 \to \sD  \to T_X \to f^* T_{X^{[1]}} \to F_{X}^* \sD  \to 0 $
\end{itemize}
and that $f^* \sK =  T_X/\sD $.
In particular, $\sK$ is a numerically flat vector bundle, which is not trivial if $X$ is not separably rationally connected.

Since $\sD$ is nef, so is $F_{X}^*\sD $.
Hence $f^* T_{X^{[1]}}$ is also nef.
In particular $T_{X^{[1]}}$ is nef.

Since $X$ is rationally chain connected, so is $X^{[1]}$.
\end{proof}

\begin{corollary}[Purely inseparable modification and vector fields]\label{cor:existence:vector_field}
Assume that $X$ is not separably rationally connected.
Then there exists a positive integer $m>0$ such that $X^{[m]}$ is separably rationally connected and $T_{X^{[m]}}$ is nef.
Moreover, $T_{X^{[m]}}$ contains an involutive, $p$-closed trivial subbundle $\bigoplus \cO$.

In particular, there exists a nowhere vanishing vector field $D $ such that $D ^p=D$ or $D^p=0$.
\end{corollary}

\begin{proof}
The first assertion follows from Proposition~\ref{prop:num_flat_subbundle} and Theorem~\ref{thm:Shen}.
By Proposition~\ref{prop:num_flat_subbundle}, there exists a numerical trivial subbundle $\sK$ of $T_{X^{[m]}}$.
By Proposition~\ref{thm:SRC:bundle:trivial}, $\sK$ is trivial.
By Remark~\ref{rem:foliation}, $\sK$ is $p$-closed and involutive and the second assertion follows.
The last assertion follows from Lemma~\ref{lem:p-closed}
\end{proof}

The above corollary yields a contradiction to the following proposition:
\begin{proposition}[Fixed points on SRC varieties]\label{prop:SRC:p-closed_vector_field}
Let $X$ be a smooth separably rationally connected variety.
Then any vector field $D$ with $D^p =D$ or $D^p=0$ admits a zero point. 
\end{proposition}

\begin{proof}
The proof proceeds as Koll\'ar's proof of simple connectedness of SRC varieties (cf. \cite[Corollaire 3.6]{Deb03}).

There is an action of $G = \mu_p$ or $\alpha_p$ that corresponds to $D$.
Fix an action of $G$ onto $\bP^1$.
Note that the set of $G$-fixed points on $\bP^1$ is not empty.

Consider the action of $G$ on $X \times \bP^1$ and the quotient $(X \times \bP^1)/G$.
Then we have the following commutative diagram:
\[
\xymatrix{
X \times \bP^1 \ar[d]_-{\pr_2}  \ar[r]^-{\varphi} &  (X\times \bP^1)/G \ar[d]^-{\pi} \\
\bP^1 \ar[r]^-{\psi} &  \bP^1/G \simeq \bP^1
}
\]
Since the general $\pi$-fiber is isomorphic to a separably rationally connected variety $X$, $\pi$ admits a section $t \colon \bP^1 \to (X\times \bP^1)/G$ \cite{GHS03,DJS03}.

Then the pull-back of this section by $\psi$ gives a $G$-equivariant  section $s \colon \bP^1 \to X \times \bP^1$ of $\pr_2$.
 Hence $\pr_1 \circ s$ is also $G$-equivariant.
Therefore we have a $G$-equivariant morphism $\bP^1 \to X$.
Since any $G$-action on $\bP^1$ admits a $G$-fiixed point, there exists a $G$-fixed point on $X$.
\end{proof}

\begin{proof}[Proof of Theorem~\ref{thm:RCC=>SRC}]
The assertion follows from Corollary~\ref{cor:existence:vector_field} and Proposition~\ref{prop:SRC:p-closed_vector_field}.
\end{proof}

\begin{proof}[Proof of Corollary~\ref{cor:SRC}]
The first assertion follows from a theorem of Koll\'ar \cite[Corollaire 3.6]{Deb03} or \cite[Corollary~5.3]{She10}.
The second assertion follows from Corollary~\ref{cor:SRC:h1}.
The last assertion follows from Theorem~\ref{thm:SRC:bundle:trivial}.
\end{proof}

\section{Existence and smoothness of extremal contractions}\label{sect:contractions}
In this section, we will prove Theorem~\ref{thm:contraction}.
Throughout this section,
\begin{itemize}
\item $X$ is a smooth projective variety over $k$, and the tangent bundle $T_X$ is nef.
\end{itemize}

We will divide the proof into several steps.

\begin{lemma}\label{lem:etale_descent}
Let $M \to N$ be a finite \'etale Galois morphism between smooth projective varieties over $k$ and $G$ its Galois group.
Assume that $M$ admits a $G$-equivariant contraction $f \colon M \to S$.
Then there exist a normal projective variety $T$, a finite  morphism $S \to T$ and a contraction $g \colon N \to T$  such that the following diagram commutes:
\[
\xymatrix{
M \ar[d]_{f} \ar[r]^{{\text{\'etale}}}  & N \ar[d]^{g} \\
S \ar[r]_{{\text{finite}}}   & T \\
}
\]
In particular, if $f$ is smooth, then any fiber of $g$ is irreducible and is the image of an $f$-fiber.
\end{lemma}

\begin{proof}
Consider the quotient map $S \to S/G$.
Then there exists a morphism $N \to S/G$ since $N = M/G$.
Let $N \to T \to S/G$ be the Stein factorization of the map $N \to S/G$.
Then, by the rigidity lemma \cite[Lemma 1.15]{Deb01}, there exists a morphism $S \to T$ as desired.
\end{proof}

In the following we consider a diagram of the following form:
\begin{equation}\label{eq:diagram}
\vcenter{
\xymatrix{
\sU  \ar[d]_{p} \ar[r]^{q} &X \\
\sM,
}}
\end{equation}
where $p$ is a smooth $\bP^1$-fibration, $\sU$ and $\sM$ are \emph{projective} varieties, and any $p$-fiber is not contracted to a point by $q$.
For example, an unsplit family of rational curves gives a diagram as above.
Two closed points $x$, $y \in X$ are said to be \emph{$\sM$-equivalent} if there exists a connected chain of rational curves parametrized by $\sM$ which contains both $x$ and $y$.
Given a set $\{ (\sM_i \xleftarrow{p_i} \sU_i \xrightarrow{q_i} X)\}_{i=1,\dots,m}$ of diagrams as above, we similarly define the $(\sM_1, \dots, \sM_m)$-equivalence relation as follows:
Two closed points $x$, $y \in X$ are said to be \emph{$(\sM_1, \dots, \sM_m)$-equivalent} if there exists a connected chain of rational curves parametrized by $\sM_1 \cup \dots \cup\sM_m$ which contains both $x$ and $y$.

Note that, if the diagram is given by a family of unsplit rational curves, then $q$ is a smooth morphism and $\sU$ and $\sM$ are smooth projective varieties by \cite[Chapter~II. Theorem~1.7, Proposition~2.14.1, Theorem~2.15, Corollary~3.5.3]{Kol96}.

\begin{proposition}[Existence of contractions]\label{prop:existence_contraction}
Let $C$ be a curve parametrized by $\sM$ and set $R \coloneqq \bR_{\geq 0} [C] \subset \cNE(X)$.
If $q$ is equidimensional with irreducible fibers, then $R$ is an extremal ray and there exists a contraction of $R$.

Moreover the contraction is equidimensional with irreducible fibers, any fiber $F$ of the contraction (with its reduced structure) is an $\sM$-equivalent class, and $\rho(F) = 1$.
\end{proposition}

\begin{proof}
Applying the same argument as in \cite[Theorem~2.2]{Kan18Kequiv}, we obtain a projective morphism $f \colon X \to Y$ onto a projective normal variety $Y$ such that each fiber is an $\sM$-equivalent class; moreover $f$ is equidimensional with irreducible fibers.
Let $F$ be a fiber of $f$.
Then, by \cite[Chapter~IV. Proposition~3.13.3]{Kol96}, the group of rational equivalence classes of algebraic $1$-cycles with rational coefficients $A_1(F)_{\bQ}$ is generated by curves in $\sM$.
Since fibers of $q$ are connected, $N_1(F) \simeq \bR$.
In particular, $R$ is an extremal ray and $f$ is the contraction of $R$.
\end{proof}

\begin{proposition}[Smoothness of contractions]\label{prop:smoothness_contraction}
Let $C$ be a curve parametrized by $\sM$ and set $R \coloneqq \bR_{\geq 0} [C] \subset \cNE(X)$.

Assume that $R$ is extremal and  the contraction $f \colon X \to Y$ of $R$ exists.
Assume moreover that
\begin{itemize}
 \item $q$ is smooth;
 \item $f$ is equidimensional with irreducible fibers;
 \item any $f$-fiber is an $\sM$-equivalent class.
\end{itemize}

Then the following hold:
\begin{enumerate}
 \item $f$ is smooth;
 \item any fiber $F$ of $f$ is an SRC Fano variety with nef tangent bundle;
 \item $T_Y$ is again nef.
\end{enumerate}
\end{proposition}

\begin{proof}
 Let $F$ be a scheme-theoretic fiber of $f$ and $F_\red$ the reduced scheme associated to $F$.

Note that $F_{\red}$ is an $\sM$-equivalent class.
By the same argument as in \cite[Lemma~4.12]{SW04}, one can show that $F_{\red}$ is smooth, and the normal bundle $N_{F_{\red}/X}$ is numerically flat.
Consider the standard exact sequence
\begin{align*}
0 \to T_{F_{\red}}\to T_X|_{F_{\red}}\to N_{F_{\red}/X}\to 0.
\end{align*}
Combining this sequence with Proposition~\ref{prop:nef_bundle}~\ref{prop:nef_bundle4}, we see that the tangent bundle $T_{F_{\red}}$ is nef.
Also, by adjunction, $-K_{F_{\red}} \equiv-K_X|_{F_{\red}}$.
Thus $F_{\red}$ is a smooth Fano variety with nef tangent bundle.
Then Theorem~\ref{thm:RCC=>SRC} implies that $F_{\red}$ is separably rationally connected. Applying Theorem~\ref{thm:SRC:bundle:trivial} and Corollary~\ref{cor:SRC:h1}, we see that $N_{F_{\red}/X}$ is trivial and $H^1(F_{\red},\cO_{F_{\red}})=0$.
In particular, the Hilbert scheme is smooth of dimension $\dim Y$ at $[F_{\red}]$.

Assume for a moment that $f$ is generically smooth.
Then we can conclude as follows;
Since $F_{\red}$ is unobstructed, it is numerically equivalent to a general fiber which is reduced by the assumption.
This is possible if and only if $F$ is generically reduced.
By \cite[Chapter~I. Theorem~6.5]{Kol96}, $f$ is smooth.
Finally, $T_Y$ is nef by Proposition \ref{prop:nef_bundle}.

Thus it is enough to prove that $f$ is generically smooth (cf.\ \cite{Sta06}).
Let $X_2 = X$ be a copy of $X$ and denote by $X_1$ the original variety $X$.
Then by taking product with $X_2$,  we have the following diagram:
\[
\xymatrix{
\sU \times X_2 \ar[d]_{P \coloneqq p \times \id} \ar[r]^{Q \coloneqq q\times \id  } &X_1 \times X_2\\
\sM \times X_2.
}
\]
Set $V_0 \coloneqq \Delta \subset X_1 \times X_2$ (the diagonal) and $V_{i+1} \coloneqq  Q(P^{-1}(P(Q^{-1}(V_{i}))))$ inductively.
Then, for $m \gg 0$, $V_m = V_{m+1} = \cdots =V$ is the graph of the $\sM$-equivalence relation (with its reduced scheme structure).
Denote by $V(x)$ the $\sM$-equivalence class of $x \in X$, which is the fiber of $V/X_2$ over the point $x$.
We will denote by $p_1 \colon V \to X_1$ and $p_2 \colon V \to X_2$ the natural projections respectively.

Note that $P$ and $Q$ are smooth, and that taking scheme theoretic image commutes with flat base changes.
Thus we see that the geometric generic fiber of  $V /X_2$ is reduced (cf.\ \cite[Lemma~2.3]{Sta06}).
Note that, by \cite[Chapter~I. Theorem~6.5]{Kol96}, $p_2$ is smooth at $(x,y) \in V$ if and only if $p_2^{-1}(y)$ is generically reduced.
In particular, $p_2^{-1}(y)$ is generically reduced if and only if  $p_2^{-1}(y)$ is reduced.

Set
\[
V^0 \coloneqq \{(x,y) \in V  \mid \text{$p_2^{-1}(y)$ is reduced} \} = \{(x,y) \in V  \mid \text{$p_1^{-1}(x)$ is reduced} \}.
\]
Then $X_i^0\coloneqq p_i(V^0)$ is open in $X_i$ and $p_i^{-1}(X_i^0) = V^0$.
Note that the subset $Y^0 \coloneqq f(X_i^0)$ is independent of $i=1$, $2$ because of the symmetry.
Moreover, $Y^0$ is open in $Y$ and $X_i^0 = f^{-1}(Y^0)$ since $X_i \setminus X_i^0$ is a closed subset, which is a union of $\sM$-equivalent classes.

Now, since $V^0/X_2^0$ is a smooth projective family of subschemes in $X_1$, we have a morphism $X_2^0 \to \Hilb(X_1)$.
Then this map factors $Y^0$ by the rigidity lemma \cite[Lemma 1.15]{Deb01}.

Consider the composite $\gamma \colon p_2^* T_{X_2^0}|_{V^0} \to T_{X _1^0 \times X_2^0}|_{V^0} \to N_{V^0/X_1^0 \times X_2^0}$ of natural homomorphisms.
Then the restriction of $\gamma$ to $p_1^{-1}(x_1)$ gives the surjection
\[
T_{X_2^0}|_{p_1^{-1}(x_1)} \to N_{p_1^{-1}(x_1)/X_2^0}.
\]
Thus $\gamma$ is surjective.
On the other hand, if we restricts $\gamma$ to $p_2^{-1}(x_2)$, we have a surjection
\[
T_{X_2^0}|_{p_2^{-1}(x_2)} \to N_{p_2^{-1}(x_2)/X_1^0}
\]
between trivial vector bundles.
Hence, by taking global sections, we have a surjection
\[
T_{X_2^0} \otimes k(x_2) \simeq H^0(T_{X_2^0}|_{p_2^{-1}(x_2)}) \to H^0(N_{p_2^{-1}(x_2)/X_1^0}) \simeq T_{\Hilb(X_1)}\otimes k([p_2^{-1}(x_2)])
\]
of vector spaces, and hence the map $X_2^0 \to  \Hilb(X_1)$ is a smooth morphism.
Recall that we have the factorization  $X_2^0 \to Y^0 \to  \Hilb(X_1)$ with a quasi-finite morphism $Y^0 \to \Hilb(X_1)$.
Then, since $X^0_2 \to Y^0$ is generically flat, the morphism $Y^0 \to \Hilb(X_1)$ is generically \'etale.
This implies that $X^0_2 \to Y^0$ is generically smooth.
\end{proof}

\begin{proposition}\label{prop:unsplit=>extremal}
 Let $\sM$ be a family as in the diagram \eqref{eq:diagram} and $[C] \in \sM$ be a curve parametrized by $\sM$.
Assume that $q$ is smooth.

Then $R \coloneqq \bR_{\geq 0} [C]$ is an extremal ray of $X$ and there exists a contraction of $R$ satisfying
\begin{enumerate}
 \item $f$ is smooth;
 \item any fiber $F$ of $f$ is an SRC Fano variety with nef tangent bundle;
 \item $T_Y$ is again nef.
\end{enumerate}
\end{proposition}

\begin{proof}
 Let $\sU \xrightarrow{q'} X' \xrightarrow{\alpha} X$ be the Stein factorization of $q$. Since $q$ is smooth, so is $\alpha$ (see for instance \cite[7.8.10 (i)]{Gro63}).
Then we see that $q'$ is also smooth (with irreducible fibers).

First we reduce to the case that the covering $\alpha$ is Galois.
There exists a finite \'etale Galois cover $X'' \to X'$ such that the composite $X'' \to X' \to X$ is also a finite Galois \'etale cover (see e.g.\ \cite[Proposition~5.3.9]{Sza09}).
Set $\sU'' \coloneqq \sU \times_{X'} X''$ and $\sU'' \to \sM'' \to \sM$ be the Stein factorization of $\sU'' \to \sM$.
Then the morphism $\sU'' \to \sM''$ is a smooth $\bP^1$-fibration and we have the following diagram:
\[
\xymatrix{
\sU''  \ar[d]_{p''} \ar[r]^{q''} &X \\
\sM''.
}
\]
Note that this diagram defines the same ray $R$ and the $\sM''$-equivalent relation is the $\sM$-equivalent relation.
Moreover the Stein factorization of $q''$ gives the Galois covering $X'' \to X$.

Therefore we may assume that the map $\alpha$ is a finite \emph{Galois} \'etale cover with Galois group $G \coloneqq \{g_1, \dots , g_m\}$.
By composing $g_i$ with $q'$, we have the following diagrams:
\[
\xymatrix{
\sU_i (=\sU)  \ar[d]_{p_i:=p} \ar[r]^{g_i \circ q'} &X' \\
\sM_i (=\sM).
}
\]
By Proposition~\ref{prop:existence_contraction}, there exist the contractions of extremal rays $R_i \coloneqq (g_i)_*R$.

Let $f_1 \colon X' \to Y_1$ be the contraction of $R_1$.
Then $f_1$ is a smooth morphism by Proposition~\ref{prop:smoothness_contraction}.
Thus the following diagram
\[
\xymatrix{
\sU_i  \ar[d]_{p_i} \ar[r]^{f_1 \circ g_i \circ q'} &Y_1 \\
\sM_i.
}
\]
satisfies the assumption of Propositions~\ref{prop:existence_contraction} and \ref{prop:smoothness_contraction}, provided that $p_i$-fibers are not contracted to a point in $Y_1$.
Thus, by repeating this procedure, we have a sequence of morphism $X' \to Y_1 \to \cdots \to Y_l$ such that each morphism is a smooth contraction whose fibers are $\sM_j$-equivalent classes for some $j$ and that all rays $R_i$ are contracted by $g \colon X' \to Y_l$.

Since  each fiber of $g$ is chain connected by curves in the families $\sM_i$, each fiber of $g$ is an $(\sM_1, \cdots , \sM_m)$-equivalence class.
Since the set of diagrams $\{ (\sM_i \xleftarrow{p_i} \sU_i \xrightarrow{q_i} X)\}_{i=1,\dots,m}$ is $G$-invariant, the $g_i$-image ($g_i \in G$) of an $(\sM_1, \cdots , \sM_m)$-equivalence class is again an $(\sM_1, \cdots , \sM_m)$-equivalence class.
This implies that there exists a $G$-action on $Y$ such that $g$ is $G$-equivariant.
By Lemma~\ref{lem:etale_descent}, there exists a contraction $f \colon X \to Y$ such that the following diagram commutes:
\[
\xymatrix{
X' \ar[d]_{g} \ar[r]^{{\text{\'etale}}}  & X \ar[d]^{f} \\
Y_l \ar[r]_{{\text{finite}}}   & Y. \\
}
\]
Since $g$ is smooth, any fiber of $f$ is irreducible and the image of a $g$-fiber (Lemma~\ref{lem:etale_descent}).
Thus each $f$-fiber is an $\sM$-equivalent class and $f$ is equidimensional with irreducible fibers.
Hence $f$ satisfies the assumption of Proposition~\ref{prop:smoothness_contraction} and the assertions follow.
\end{proof}

\begin{remark}\label{rem:cont:RCC}
Assume $X$ is RCC, then $X$ is simply connected.
Thus the map $\sU \to X$ has connected fibers.
In this case, by Proposition~\ref{prop:existence_contraction}, we have $\rho(F) = 1$ for any fiber $F$ of the extremal contraction
\end{remark}

\begin{proof}[Proof of Theorem~\ref{thm:contraction}]
Take a rational curve $C$ on $X$ such that $R=\bR_{\geq 0}[C]$ and the anticanonical degree of $C$ is minimal among rational curves in the ray $R$. Let $\sM$ be the unsplit family of rational curves containing $[C]$:
\[
\xymatrix{
\sU  \ar[d]_{p} \ar[r]^{q} &X \\
\sM.
}
\]

Then $q$ is smooth and the assertion follows from Proposition~\ref{prop:unsplit=>extremal}
\end{proof}

\section{Rational chain connectedness and positivity of $-K_X$}\label{sect:Fano}
In this section, we will prove Theorem~\ref{thm:RCC=>Fano}.
The proof proceeds as \cite[Proof of Theorem 4.16]{Wat20}.
Here we divide the proof into several steps.

Throughout this section,
\begin{itemize}
 \item $X$ denotes a smooth projective variety with nef tangent bundle and $f \colon X \to Y$ is a smooth contraction with RCC fibers.
\end{itemize}
Note that by Theorem~\ref{thm:RCC=>SRC} all fibers are separably rationally connected.

\begin{lemma}[Correspondence of extremal rays]\label{lem:correspondence_ext_ray}
Suppose that $R_Y$ is an extremal ray of $Y$.
Then there exists an extremal ray $R_X$ of $X$ such that $f_*R _X =R_Y$.

On the other hand, if $R_X$ is an extremal ray of $X$, then $f_* R_X$ is an extremal ray unless $f_* R_X = 0$.
\end{lemma}

\begin{proof}
Let $C_Y$ be a rational curve such that the class belongs to $R_Y$ and $-K_Y \cdot C_Y$ is minimum.
Since any fiber of $f$ is RCC, it is SRC by Theorem~\ref{thm:RCC=>SRC}.
Thus, by \cite{GHS03,DJS03}, there exists a rational curve $C_X$ on $X$ such that $f|_{C_X} \colon C_X \to C_Y$ is birational.
Take such a rational curve $C_X$ with minimum anti-canonical degree $-K_X \cdot C_X$ (note that, since $C_X$ is free, the degree is at least $2$).

Then the family of rational curves of $C_X$ is unsplit.
Thus, by Proposition~\ref{prop:unsplit=>extremal}, $\bR_{\geq 0} [C_X]$ defines an extremal ray of $X$.

Conversely, assume that $R_X$ is an extremal ray of $X$.
Let $\sM$ be a family of minimal rational curves in $R_X$ and let
\[
\xymatrix{
\sU \ar[d]_{p} \ar[r]^{q} &X \\
\sM
}
\]
be the diagram of this family of rational curves.
Then $f \circ q$ gives a diagram satisfying the assumption of Proposition~\ref{prop:unsplit=>extremal}, and hence $f_*R_X$ is an extremal ray.
\end{proof}

\begin{proposition}[Relative Picard numbers]\label{prop:Picard_number}
Let $Y \to Z$ be another contraction.
Then $\rho(X/Y) =\rho (X/Z) - \rho (Y/Z)$ and we have the following exact sequence:
\[
0 \to N_1(X/Y) \to N_1(X/Z) \to N_1(Y/Z) \to 0.
\]
\end{proposition}

\begin{proof}
Note that any $f$-fiber $F$ is SRC and, in particular, any numerically trivial line bundle on $F$ is trivial (Corollary \ref{cor:SRC:h1}).

Thus the sequence
\[
0 \to N^1(Y/Z) \to N^1(X/Z) \to N^1(X/Y) \to 0
\]
is exact and the assertions follow.
\end{proof}

\begin{theorem}[Relative Kleiman-Mori cone]\label{thm:simplicial}
Fix an integer $m$, then the following are equivalent:
\begin{enumerate}
 \item \label{thm:simplicial1} $\rho (X/Y) =m$.
 \item \label{thm:simplicial2} There exists a sequence of smooth contractions of $K$-negative extremal rays 
 \[
 X=X_0 \to X_1 \to X_2 \to \cdots \to X_{m-1} \to X_{m} = Y.
 \]
\item \label{thm:simplicial3} $\NE (X/Y)$ is a closed simplicial cone of dimension $m$, and it is generated by $K_X$-negative extremal rays.
\end{enumerate}
In particular, any fiber $F$ is a smooth Fano variety.
\end{theorem}

\begin{proof}
The proof proceeds by induction on $m$.
If $m =1$, then the equivalence of the three conditions follows from Theorem~\ref{thm:contraction}.

Assume $m \geq 2$.
Note that (\ref{thm:simplicial3} $\implies$ \ref{thm:simplicial1}) is trivial.
Also note that Proposition~\ref{prop:Picard_number} imply (\ref{thm:simplicial2} $\implies$ \ref{thm:simplicial1}).

Now, we prove (\ref{thm:simplicial1} $\implies$ \ref{thm:simplicial2}).
Since $K_X$ is not $f$-nef, there exists a $K_X$-negative extremal ray $R \subset \cNE (X/Y)$.
By Theorem~\ref{thm:contraction}, there exists a smooth contraction $f_1 \colon X \to X_1$ of $R$, and $f$ factors through $X_1$:
\[
X \xrightarrow{f_1} X_1 \to Y.
\]
Since $f_1$ is smooth and any fiber of $f$ is RCC,
the morphism $X_1 \to Y$ is also smooth and fibers are RCC.
Since $\rho (X_1/Y)<m$, the contraction $X_1 \to Y$ satisfies the three conditions.
Hence \ref{thm:simplicial2} holds for $f \colon X \to Y$.

Finally we prove \ref{thm:simplicial2} $\implies$  \ref{thm:simplicial3}.
Since $\rho(X/X_{m-1}) <m$ by Proposition~\ref{prop:Picard_number}, the contraction $g \colon X \to X_{m-1}$ satisfies the three conditions.
In particular, $\NE(X/{X_{m-1}})$ is a simplicial $K_X$-negative face of dimension $m-1$.
Denote by $R_1 =\bR_{\geq 0} [C_1]$, \dots, $R_{m-1}=\bR_{\geq 0} [C_{m-1}]$ the extremal rays of  $\NE(X/{X_{m-1}})$, which are spanned by curves $C_i$.
Note also that $X_{m-1} \to Y$ is a contraction of an extremal ray $R_{X_{m-1}}$.

Then, by Lemma~\ref{lem:correspondence_ext_ray}, there exists an extremal ray $R$ of $X$ such that $g_* R = R_{X_{m-1}}$.
Let $h \colon X \to Z$ be the contraction of $R$, then we have the following commutative diagram:
\[
\xymatrix{
X \ar[r]^{g} \ar[d]_{h} & X_{m-1}\ar[d]^{u}\\
Z\ar[r]_{v} & Y.
}
\]

Note that, by the inductive hypothesis, each morphism $g$, $h$, $u$, $v$ satisfies the three conditions.
Since $N_1(X/X_{m-1}) \to N_1(Z/Y)$ is a surjective homomorphism between the same dimensional vector spaces, it is an isomorphism.

By Lemma~\ref{lem:correspondence_ext_ray}, the map $N_1(X/X_{m-1}) \to N_1(Z/Y)$ sends each extremal ray to an extremal ray.
Thus $\cNE(X/X_{m-1}) \simeq \cNE(Z/Y)$ since these two cones are simplicial.
In particular $\cNE(Z/Y)$ is spanned by $h_*(C_i)$ ($i=1$, \dots $m-1$).

Let $C$ be an irreducible curve such that $f(C)$ is a point.
Then $h_*(C) \in \NE(Z/Y)$.
Thus we may write
\[
h_*(C) = \sum_{i=1}^{m-1} a_i h_*(C_i)
\]
with non-negative real numbers $a_i$.
Thus $C - \sum_{i=1}^{m-1} a_i C_i \in \Ker(h_*) = N_1(X/Z)$.
Since $g^*H \cdot (C - \sum_{i=1}^{m-1} a_i C_i) \geq 0$ for any ample divisor $H$ on $X_{m-1}$, we see that
\[
C - \sum_{i=1}^{m-1} a_i C_i \in  R.
\]
This proves $\cNE(X/Y)$ is spanned by $R$, $R_1$, \dots, $R_{m-1}$, and hence it is simplicial (since $\rho(X/Y) =m$).
\end{proof}

\begin{remark}\label{rem:cone}
If $X$ is RCC, or equivalently, a Fano variety, then each $X_i$ is also a Fano variety.
Thus, by Remark~\ref{rem:cont:RCC}, the fibers of $X_i \to X_{i+1}$ have Picard number one.
Then, by arguing as above, we can show that each fiber $F$ of the contraction $X \to  Y$ is a Fano variety with $\rho(F) = m$.
Moreover $N_1(F) $ and $\NE(F)$ are identified with $N_1(X/Y)$ and $\NE(X/Y)$ respectively via the inclusion.
\end{remark}

\section{Decomposition of varieties with nef tangent bundles}\label{sect:decomposition}
Here we will prove Theorem~\ref{thm:decomposition}.

\begin{proof}
If $K_X$ is nef, then $T_X$ is numerically trivial and $\varphi =\id \colon X \to X$ gives the desired contraction.

Assume that $K_X$ is not nef.
Then, by Mori's cone theorem, we can find an extremal ray $R \subset \cNE(X)$ of $X$.
Then, by Theorem~\ref{thm:contraction}, there exists a smooth contraction $f \colon X \to Y$ of $R$ and the tangent bundle of $Y$ is again a nef.
Then, by induction on the dimension, $Y$ admits a smooth contraction $\varphi'  \colon Y \to M$ satisfying the conditions in Theorem~\ref{thm:decomposition}.
Set $\varphi \coloneqq \varphi' \circ f$.

Let $F$ be a fiber of $\varphi$.
Then $T_F$ is nef.
Moreover, since any fiber of $f$ is separably rationally connected and any fiber of $\varphi'$ is rationally chain connected, it follows that $F$ is rationally chain connected \cite{GHS03,DJS03}.
Hence $F$ is a smooth Fano variety by Theorem~\ref{thm:RCC=>Fano}.

Since $T_M$ is numerically flat, there are no rational curves on $M$.
Thus $\varphi\colon X \to M$ is the MRCC fibration of $X$.
\end{proof}

\begin{corollary}[Contraction of extremal faces]\label{cor:face}
 Let $X$ be a smooth projective variety with nef tangent bundle and $\varphi \colon X \to M$ be the decomposition morphism as in Theorem~\ref{thm:decomposition}.
Then $\NE(X/M)$ is simplicial.
Any set of $K_X$-negative extremal rays spans an extremal face.

Moreover, for any $K_X$-negative extremal face $F$, there exists the contraction of $F$ and it is a smooth morphism satisfies the three equivalent conditions in Theorem~\ref{thm:simplicial}.
\end{corollary}

\section{$F$-liftable varieties with nef tangent bundles}
Here we apply our results to prove Theorem~\ref{thm:F-liftable}

\subsection{Preliminaries on $F$-liftability}

\begin{definition}[$F$-liftable varieties]\label{def:f-liftable}
Let $X$ be a projective variety over $k$. 
\begin{enumerate}
\item A \emph{lifting of $X$ (modulo $p^2$)} is a flat scheme $\widetilde{X}$ over the ring $W_2(k)$ of Witt vectors of length two with $\widetilde{X}\times_{{\Spec} W_2(k)} \Spec (k) \cong X$.
\item For such a lifting $\widetilde{X}$ of $X$, \emph{a lifting of Frobenius} on $X$ to $\widetilde{X}$ is a morphism $\widetilde{F_X} \colon \widetilde{X} \to \widetilde{X}$ such that the restriction $\widetilde{F_X}|_X$ coincides with the Frobenius morphism $F_X$;
then the pair $(\widetilde{X}, \widetilde{F_X})$ is called a \emph{Frobenius lifting of $X$}.
If there exists such a pair, $X$ is said to be \emph{F-liftable}.
\end{enumerate}
\end{definition}

\begin{proposition}[$F$-liftable varieties]\label{prop:f-liftable}
For a smooth $F$-liftable projective variety $X$, the following hold:
\begin{enumerate}
\item \label{prop:f-liftable1} For any finite \'etale cover $Y \to X$, $Y$ is also F-liftable.
\item \label{prop:f-liftable2} Let $f \colon  X \to Y$ be a contraction and assume $R^1f_{\ast}\cO_X=0$.
Then $Y$ and any fiber $F$ of $f$ are also F-liftable.
\end{enumerate}
\end{proposition}

\begin{proof}
\ref{prop:f-liftable1} follows from \cite[Lemma~3.3.5]{AWZ17}.
\ref{prop:f-liftable2} follows from \cite[Theorem~3.3.6~(b)]{AWZ17} and \cite[Corollary~3.5.4]{AWZ17}.
\end{proof}

\subsection{Proof of Theorem~\ref{thm:F-liftable}}

\begin{definition}[$F$-liftable and nef tangent bundle]\label{def:FLNT}
Let $X$ be a smooth projective variety. For convenience, let us introduce the following condition:
\begin{enumerate}
\renewcommand{\labelenumi}{\rm{(FLNT)}}
\item \label{cond:FLNT} $X$ is $F$-liftable and the tangent bundle is nef.
\end{enumerate}
\end{definition}

\begin{proposition}[{\cite[Proposition~6.3.2]{AWZ17}}]\label{prop:f-liftable:rho=1}
Let $X$ be a smooth Fano variety satisfying the condition \ref{cond:FLNT}.
If $\rho_X=1$, then $X$ is isomorphic to a projective space.
\end{proposition}

\begin{proposition}[$F$-liftable varieties with $K_X \equiv 0$]\label{prop:K=0}
Let $X$ be a smooth projective $F$-liftable variety.
If the canonical divisor $K_X$ is numerically trivial, then there exists a finite \'etale cover $f\colon Y \to X$ from an ordinary abelian variety $Y$.
\end{proposition}

\begin{proof}
See for instance \cite[Theorem~5.1.1]{AWZ17}.
\end{proof}

\begin{proposition}[FLNT Fano varieties]\label{prop:Fano:product}
Let $X$ be a smooth Fano variety satisfying the condition \ref{cond:FLNT}.
Then $X$ is isomorphic to a product of projective spaces.
\end{proposition}

\begin{proof}
We proceed by induction on the Picard number $\rho_X$.
By Proposition~\ref{prop:f-liftable:rho=1}, our assertion holds for the case $\rho_X=1$.
Assume that $\rho_X \geq 2$.
Then there exists a two dimensional extremal face, which is spanned by two extremal rays $R_1$ and $R_2$.
We denote the contraction of the extremal ray $R_i$ by $f_i\colon  X \to X_i$ ($i=1$, $2$), which is a smooth $\bP$-fibration by Theorem~\ref{thm:contraction}, Remark~\ref{rem:cont:RCC} and Proposition~\ref{prop:f-liftable:rho=1}.
By the induction hypothesis, each $X_i$ is a product of projective spaces and hence the Brauer group of $X_i$ vanishes;
this implies that each $f_i$ is given by a projectivization of a vector bundle.
When $\rho_X=2$, applying \cite[Theorem~A]{Sat85} $X$ is isomorphic to a product of two projective spaces or $\bP(T_{\bP^n})$.
However the latter does not occur, because $\bP(T_{\bP^n})$ is not F-liftable by \cite[Lemma~6.4.3]{AWZ17}.
Thus we may assume that $\rho_X \geq 3$.
Let $\pi \colon X \to X_{1,2}$ be the contraction of the extremal face $R_1+R_2$.
By the rigidity lemma \cite[Lemma 1.15]{Deb01}, there is the following commutative diagram:
\[
\xymatrix{
& X \ar[dd]_-{\pi} \ar[dl]_-{f_1} \ar[dr]^-{f_2} & \\
X_1 \ar[dr]_-{g_1} & & X_2 \ar[dl]^-{g_2} \\
&X_{1,2}& \\
}
\]

Let us take a vector bundle $\sE$ on $X_1$ such that $X \cong \bP(\sE)$ and $f_1\colon  X \to X_1$ is given by the natural projection $\bP(\sE) \to X_1$.
For any point $p \in X_{1,2}$, $\pi^{-1}(p)$ is a smooth Fano variety of Picard number two satisfying the condition \ref{cond:FLNT};
by the induction hypothesis it is a product of two projective spaces.
Thus, by tensoring $\sE$ with a line bundle, we may assume that $\sE|_{g_1^{-1}(p)}$ is trivial for any point $p \in X_{1,2}$.
By Grauert's theorem \cite[III. Corollary~12.9]{Har77}, we see that ${g_1}_\ast(\sE)$ is a vector bundle on $X_{1,2}$.
Then it is straightforward to verify that the natural map ${g_1}^{\ast}({g_1}_\ast(\sE)) \to \sE$ is an isomorphism via Nakayama's lemma. Thus we have 
\[
X=\bP(\sE)\cong \bP({g_1}^{\ast}({g_1}_\ast(\sE)))\cong X_1\times_{X_{1,2}} \bP({g_1}_\ast(\sE)).
\]
Again, by the inductive hypothesis, $X_1\cong \prod_{j=1}^{\rho_X-1} \bP^{n_j}$ and $X_{1,2}\cong \prod_{j=2}^{\rho_X-1} \bP^{n_j}$;
we also see that $X_1 \to X_{1,2}$ is a natural projection.
This concludes that $X$ is isomorphic to $\bP^{n_1}\times \bP({g_1}_\ast(\sE))$.
Since $\bP({g_1}_\ast(\sE))$ is a product of projective spaces, our assertion holds.
\end{proof}

\begin{corollary}[Structure theorem of FLNT varieties]\label{cor:str:them:1}
Let $X$ be a smooth projective variety satisfying the condition \ref{cond:FLNT}.
Then there exists a $K_X$-negative contraction $\varphi \colon  X \to M$ which satisfy the following properties:
\begin{enumerate}
 \item $M$ is an \'etale quotient of an ordinary abelian variety $A$.
 \item $\varphi$ is a smooth morphism whose fibers are isomorphic to a product of projective spaces.
\end{enumerate}
In particular, there exist an finite \'etale cover $Y \to X$ and a smooth contraction $f \colon Y \to A$ onto an ordinary abelian variety $A$ such that all $f$-fibers are isomorphic to a product of projective spaces.
\end{corollary}

\begin{proof}
By Theorem~\ref{thm:decomposition}, $X$ admits a smooth fibration $\varphi \colon X \to M$ such that all fibers are smooth Fano varieties and the tangent bundle of $M$ is numerically trivial.
Since Corollary~\ref{cor:SRC} implies that $H^1(F, \cO_F) = 0$ for all fiber $F$ of $\varphi$, we have $R^1\varphi _* \cO_X =0$.
Hence all fibers $F$ and the image $M$ are $F$-liftable by Proposition~\ref{prop:f-liftable}.
In particular $F$ is isomorphic to a product of projective spaces and $M$ is an \'etale quotient of an ordinary abelian variety $A$.

Set $Y \coloneqq X \times_M A$ and let $f \colon Y \to A$ be the natural projection.
Then the last assertion follows.
\end{proof}

\begin{proof}[Proof of Theorem~\ref{thm:F-liftable}]
By Corollary~\ref{cor:str:them:1}, there exist a finite \'etale cover $\tau_1\colon  Y_1 \to X$ and a smooth $(\prod \bP^{n_i})$-fibration $f_1\colon  Y_1 \to A_1$ over an ordinary abelian variety.
We may find a finite \'etale cover $\tau_2\colon  Y \to Y_1$ such that the composite $ Y \to Y_1 \to  X$ is an \'etale Galois cover (see for instance \cite[Proposition~5.3.9]{Sza09}).
We denote the Albanese morphism by $\alpha \coloneqq \alpha_{Y} \colon  Y \to A \coloneqq \Alb(Y)$.
By the universal property of the Albanese variety, we obtain a morphism $\tau_3 \colon A \to A_1$ which satisfies the following commutative diagram:
\[\xymatrix{
Y\ar[d]_{\alpha} \ar[r]^{\tau_2} & Y_1 \ar[d]^{f_1} \ar[r]^{\tau_1} & X\\
A \ar[r]_{\tau_3} & A_1 & \\
}\]

Here we prove that $\tau_3$ is \'etale.
All $f_1$-fibers are isomorphic to $\prod \bP^{n_i}$ and hence simply connected.
Thus each $f_1 \circ \tau_2$-fiber is a disjoint union of $\prod \bP^{n_i}$.
In particular, $f_1 \circ \tau_2$-fibers are contracted by $\alpha$.
Therefore $\tau_3$ is finite.
Since $f_1 \circ \tau_2$ is surjective, the morphism $ \tau_3$ is also surjective.
In particular $\tau_3$ is a finite surjective morphism  between abelian varieties.
This in turn implies that $\alpha$ is surjective with equidimensional fibers.
Thus $\alpha$ is flat.
Since $\tau_3 \circ \alpha$ is smooth, $\tau_3$ is \'etale.

Since $\tau _3$ is \'etale and $\tau_3 \circ \alpha$ is smooth, the morphism $\alpha$ is smooth.
Thus it is enough to show that $\alpha$-fibers are connected.
To prove this, let us consider the Stein factorization $Y \xrightarrow{q_1} A' \xrightarrow{q_2} A$, where $q_1$ is a contraction and $q_2$ is a finite morphism. Since $q_2$ is \'etale, $A'$ is an abelian variety (see for instance \cite[Section~18]{Mum70}). This implies that $q_1\colon  Y \to  A'$ factors through $\alpha \colon Y \to A$, that is, there exists a morphism $\beta \colon A\to  A'$ such that $q_1 =\beta \circ \alpha$.
By virtue of the universality of $A$ and the rigidity lemma \cite[Lemma 1.15]{Deb01} for $q_1$, we see that $\beta$ is an isomorphism and the assertion follows.
\end{proof}

\bibliographystyle{amsalpha}
\bibliography{biblio}
\end{document}